\pdfoutput=1
\documentclass[hidelinks, reqno]{amsart}
\usepackage[english]{babel}
\usepackage{graphicx}
\usepackage{subcaption}
\usepackage{tabularx}
\usepackage{booktabs}
\usepackage{float}
\usepackage{array}
\usepackage{amsmath}
\usepackage{amsfonts}
\usepackage{amssymb}
\usepackage{amsthm}
\usepackage[scr]{rsfso}
\usepackage{enumitem}
\usepackage{permute}
\usepackage{slashed}
\usepackage[usenames,dvipsnames]{xcolor}
\usepackage[pagebackref = true, colorlinks, linkcolor = Red, citecolor = Green, bookmarksdepth=2, linktocpage=true]{hyperref}
\usepackage[capitalize]{cleveref}
\usepackage{comment}
\usepackage{changepage}
\usepackage{dsfont}

\usepackage[T1]{fontenc}
\usepackage{makecell}

\usepackage{tikz}
\usetikzlibrary{calc}

\allowdisplaybreaks

\DeclareMathOperator{\vol}{vol}

\DeclareMathOperator{\Var}{Var}

\DeclareMathOperator{\im}{im}

\DeclareMathOperator{\Riemann}{Rm}
\DeclareMathOperator{\Ricci}{Ric}
\DeclareMathOperator{\Bl}{Bl}

\newcommand{\Z}{\mathbb{Z}}

\newcommand{\R}{\mathbb{R}}
\newcommand{\C}{\mathbb{C}}

\newcommand{\Hau}{\mathcal{H}}

\newcommand{\Proj}{\mathbb{P}}

\makeatletter
\newcommand{\mytag}[2]{%
\text{#1}%
\@bsphack
\begingroup
\@onelevel@sanitize\@currentlabelname
\edef\@currentlabelname{%
\expandafter\strip@period\@currentlabelname\relax.\relax\@@@%
}%
\protected@write\@auxout{}{%
\string\newlabel{#2}{%
{#1}%
{\thepage}%
{\@currentlabelname}%
{\@currentHref}{}%
}%
}%
\endgroup
\@esphack
}
\makeatother

\DeclareFontFamily{U}{mathb}{\hyphenchar\font45}
\DeclareFontShape{U}{mathb}{m}{n}{
<5> <6> <7> <8> <9> <10> gen * mathb
<10.95> mathb10 <12> <14.4> <17.28> <20.74> <24.88> mathb12
}{}
\DeclareSymbolFont{mathb}{U}{mathb}{m}{n}
\DeclareMathSymbol{\bigast}{2}{mathb}{"06}

\def\XXint#1#2#3{{\setbox0=\hbox{$#1{#2#3}{\int}$}
\vcenter{\hbox{$#2#3$}}\kern-.5\wd0}}

\theoremstyle{plain}
\newtheorem{theorem}{Theorem}[section]

\newtheorem{lemma}[theorem]{Lemma}
\newtheorem{corollary}[theorem]{Corollary}
\newtheorem{conjecture}[theorem]{Conjecture}
\newtheorem{claim}[theorem]{Claim}

\theoremstyle{definition}

\theoremstyle{remark}
\newtheorem{remark}[theorem]{Remark}

\Crefname{enumi}{Property}{Properties}
\Crefname{alternativei}{Alternative}{Alternatives}
\Crefname{subsection}{Subsection}{Subsections}

\numberwithin{equation}{section}

\title[]{on the Tangent flow to the collapsing K\"{a}hler-Ricci flow on Hirzebruch Surfaces}

\author{Jiangtao Li}
\address{Department of Mathematics, UC San Diego, 9500 Gilman Drive, La Jolla, CA 92093-0112}
\email{jil320@ucsd.edu}

\date{}

\begin{document}

\maketitle
\begin{abstract}
    In this paper, we study the collpasing K\"{a}hler-Ricci flow on Hirzebruch surfaces $M_k=\Proj(\mathcal{O}\oplus\mathcal{O}(k))$. We show that any tangent flow based at a spacetime point at the singular time slice is the K\"{a}hler-Ricci flow associated with a nonflat K\"{a}hler-Ricci shrinker with finitely many orbifold singularities .
\end{abstract}

\section{Introduction}

Given a compact K\"{a}hler surface $M$ with a K\"{a}hler metric $\omega_0$ on $M$. The K\"{a}hler-Ricci flow on $M$ with inital metric $\omega_0$ is the following equation:
\begin{equation}\label{eq:KahlerRicciFlow}
    \begin{cases}
        &\frac{\partial\omega}{\partial t}=-\Ricci(\omega(t)),\\
        &\omega(0)=\omega_0.
    \end{cases}
\end{equation}

Suppose that the maximal existence time of the flow is $0<T\leq\infty$. When $T<\infty$, the flow develops a \textit{finite time singularity} at $t=T$. Finite time singularities are further classified into two categories according to the blow up rate of the Riemann curvature: 
\begin{enumerate}
    \item If 
        \[\limsup_{t\to T}(T-t)\sup_{x\in M}|\Riemann(x,t)|<\infty,\] 
    the finite time singularity is a \textit{type I} singularity;
    \item If 
        \[\limsup_{t\to T}(T-t)\sup_{x\in M}|\Riemann(x,t)|=\infty,\]
    the finite time singularity is a \textit{type II} singularity.
\end{enumerate}

Understanding the limiting behaviors of Kähler-Ricci flows at finite-time singularities is an active area of research, with significant progress made over the past decades. Several studies related to this work are: (1) In \cite{CW12}, Chen and Wang examined the K\"{a}hler-Ricci flow on Fano surfaces with canonical initial metric, demonstrating that in these cases the singularity is type I and that the metric converges to a shrinking soliton metric after a type I normalization; (2) In \cite{SW11}, Song and Weinkove investigated the K\"{a}hler-Ricci flow on Hirzebruch surfaces with an initial metric satisfying the Calabi Ansatz, classifying the Gromov-Hausdorff limits at the singular time based on the cohomology class of the initial metric; (3) A recent paper \cite{BCCD24} shows that the K\"{a}hler-Ricci flow on $\Bl_p(\mathbb{P}^1\times\mathbb{P}^1)$ with a torus invariant initial metric always develops a type I singularity. Utilizing this along with an earlier result in \cite{EMT11}, the authors constructed a shrinking gradient K\"{a}hler-Ricci soliton structure implicitly on $\Bl_p(\C\times\mathbb{P}^1)$, which presents an alternative approach to the construction other than solving Monge-Amp\`{e}re equations. Together with the classification theorem in \cite{CCD22}, this completes the classification of shrinking gradient K\"{a}hler-Ricci soliton surfaces with bounded scalar curvature (the boundedness assumption was removed in \cite{LW25} later). 

Partly motivated by the work \cite{BCCD24}. The following conjecture attracts attention recently:
\begin{conjecture}
    Finite time singularities of K\"{a}hler-Ricci flow on K\"{a}hler surfaces are all type I singularities.
\end{conjecture}

This conjecture is false in dimensions higher than 2,as demonstrated by counterexamples presented in \cite{LTZ24} and \cite{MT22}. Recently, Conlon, Hallgren and Ma (cf.\cite{CHM25}) proved the above conjecture under the addition assumption that the flow is volume non-collapsing. 

Indeed, both \cite{BCCD24} and \cite{CHM25} utilized the compactness and structure theory of Ricci flows developed by Bamler in a series of papers (cf.\cite{Bam21a,Bam23,Bam21b}), which is particularly effective for dealing with Ricci flow in the absence of curvature bounds, especially in real dimension 4. These papers established their main results by examining the possible tangent flows of the K\"{a}hler-Ricci flow at the singular time. More specifically, both show that all the tangent flows are smooth under their settings and hence it follows that the singularity is type I. In this work, we adopt the same approach and analyze tangent flows of volume collapsing K\"{a}hler-Ricci flow on Hirzebruch surfaces. Our main theorem is the followng: 
\begin{theorem}[Main Theorem]\label{thm:MainTheorem}
    If $g(t),0\leq t<T$ is a K\"{a}hler-Ricci flow on the Hirzebruch surface $M_k$ which is volume collapsing, then any tangent flow at the singular time is the K\"{a}hler-Ricci flow associated to a nonflat gradient shrinking K\"{a}hler-Ricci soliton with finitely many orbifold singularities. 
\end{theorem}

The paper is organized as follows: In section 2, we review the fundamentals of K\"{a}hler-Ricci flows on Hirzebruch surfaces and some key results from Bamler's theory that will be used in the proof; In section 3, we demonstrate the nonflat part of the main theorem, specifically, we show that the tangent flow is not the constant flow on the flat quotient $\C^2/\Gamma$ (cf.\cref{thm:TangentFlowsAreNotConstant}); Finally, in section 4, we establish that there are at most finitely many orbifold singularities in the nonflat K\"{a}hler-Ricci soliton (cf.\cref{thm:FiniteSingularities}), thereby completing the proof of the main theorem.

\subsection*{Acknowledgements} The author wishes to express his sincere gratitude to Zilu Ma and Max Hallgren for their helpful discussions and suggestions.

\section{Preliminaries}
\subsection{K\"{a}hler-Ricci flow on Hirzebruch surfaces}
We review some basic facts about the K\"{a}hler-Ricci flow on the Hirzebruch surfaces. The notations are the same as in \cite{SW11}. 

Let $M_k=\Proj(\mathcal{O}\oplus\mathcal{O}(k))$ be the projectivization of the rank 2 holomorphic vector bundle $\mathcal{O}\oplus\mathcal{O}(k)$ over $\mathbb{P}^1$. $M_k$ is called \textit{the $k$-th Hirzebruch surface}. 

Suppose that $D_\infty$ is the divisor on $M_k$ corresponding to the section $(1,0)$ of $\mathcal{O}\oplus\mathcal{O}(k)$. Because of the identification $M_k\cong\Proj(\mathcal{O}(-k)\oplus\mathcal{O})$, we may let $D_0$ be the divisor corresponding to the section $(0,1)$ of $\mathcal{O}(-k)\oplus\mathcal{O}$. Then $H^{1,1}(M_k,\R)$ is spanned by $[D_0]$ and $[D_\infty]$. By the following intersection relations
\[[D_\infty]\cdot[D_\infty]=k,[D_0]\cdot[D_0]=-k,[D_0]\cdot[D_\infty]=0,\]
and the Nakai-Moishezon criterion, the K\"{a}hler cone $\mathcal{C}(M_k)$ of $M_k$ is given by
\[\mathcal{C}(M_k)=\{\frac{b}{k}[D_\infty]-\frac{a}{k}[D_0]:b>a>0\}.\]
Moreover, the first Chern class $c_1(M_k)=(k+2)/k[D_\infty]-(k-2)/k[D_0]$.

Consider the K\"{a}hler-Ricci flow \cref{eq:KahlerRicciFlow} on $M_k$. Suppose $[\omega_0]=b/k[D_\infty]-a/k[D_0]$. Then the evolution of the class of the K\"{a}hler metric along the flow is known:
\begin{equation}\label{eq:EvolutionOfTheMetric}
    [\omega(t)]=[\omega_0]-tc_1(M)=\frac{b-t(k+2)}{k}[D_\infty]-\frac{a+t(k-2)}{k}[D_0].
\end{equation}
Let $T$ be the maximal existence time of the flow. We have (cf.[Theorem 3.3.1]\cite{BEG13})
\begin{equation}\label{eq:MaximalExistenceTime}
    T=\sup\{t:[\omega(t)]-tc_1(M_k)\in\mathcal{C}(M_k)\}.
\end{equation}
Moreover, the evolution of the total volume is:
\begin{align}\label{eq:EvolutionOfTheVolume}
    &\vol_{g(t)}(M_k)=\int_M\omega(t)^2=\left(\frac{b-t(k+2)}{k}[D_\infty]-\frac{a+t(k-2)}{k}[D_0]\right)^2\nonumber\\
    = & \frac{1}{k}(a+b-4t)(b-a-2kt).
\end{align}

From \cref{eq:MaximalExistenceTime} and \cref{eq:EvolutionOfTheVolume} we have the following classification:
\begin{enumerate}
    \item If $k\ge 2$, then $T=\frac{b-a}{2k}$ and $\lim_{t\to T-}\vol_{g(t)}(M_k)=0$;
    \item If $k=1$, there are two cases:
    \begin{enumerate}
        \item If $b\leq 3a$, then $T=\frac{b-a}{2k}$ and $\lim_{t\to T-}\vol_{g(t)}(M_1)=0$;
        \item If $b>3a$, then $T=a$ and $\lim_{t\to T-}\vol_{g(t)}=(b-3a)^2> 0$.
    \end{enumerate}
\end{enumerate}

In the cases (1) and (2a), the flow develops a finite time singularity and it is \textit{volume collapsing}. In the case (2b), the flow develops a finite time singularity and is \textit{volume noncollapsing}. 

If we further assume that the initial metric satisfies the Calabi symmetry condition, the flow equation can be reduced to calculations of ODEs and the Gromov-Hausdorff limits of the above flow are completely classified in \cite{SW11}. 

As for the singularity type, in the case when the flow is volume non-collapsing, it must develop a type I singularity (cf.\cite{CHM25}) and the singularity models are either constant flow on $\C^2$ or the flow on FIK shrinker (cf.\cite{FIK03}). On the contrast, the volume collapsing case is still unknown.

Therefore, in this paper, we consider the volume collapsing K\"{a}hler-Ricci flow on $M_k$ without imposing any symmetry condition on the initial metric.

\subsection{The tangent flow to K\"{a}hler-Ricci flow at the singular time}

In this subsection, we review some of the results from \cite{Bam21a},\cite{Bam21b} and \cite{Bam23} that will be used in our proofs in later sections.

Consider a Ricci flow $(M^4,g(t)),0\leq t<T<\infty$ on a compact K\"{a}hler surface $M$, which develops a finite time singularity at $t=T$. In \cite[Section 2.6]{Bam21b}, the ``singular time slice'' $M_T$ at $t=T$ is constructed as follows:
$M_T$ consists of all probability measures $\mu_t,0\leq t<T$ satisfying the conjugate heat equation and such that $\lim_{t\to T}\Var(\mu_t)=0$. We adopt the notations from \cite{Bam21b} and let $x$ be a point in $M_T$ whose associated conjugate heat kernel is $\nu_{x,T;t}:=\mu_t$.

For Ricci flow with finite time singularities, we can always blow up and get a corresponding tangent flow, as asserted in the following theorem:
\begin{theorem}[Tangent flow, \text{\cite[Theorem 2.18]{Bam21b}}]\label{thm:TangetFlow}
    For any sequence $\tau_i>0, \tau_i\searrow 0$, the sequence of metric flow pairs 
    \[(M^4,\tau_i^{-1}g(T+\tau_it),\nu_{x,T;T+\tau_it})\]
    admits a subsequence which $\mathbb{F}$ converges to a metric soliton $(\mathcal{X},\nu_{x_\infty,0;t})_{t\in(-\infty,0]}$. Moreover, the pointed Nash entropy $\mathcal{N}_{x_\infty}(t)\equiv \mathcal{N}_{x,T}(0)=W$ is a constant for all $t\leq 0$. 
    
    Moreover, there is a regular-singular decomposition of the tangent flow $\mathcal{X}=\mathcal{R}\cup\mathcal{S}$, where the singular part $\mathcal{S}$ has $*$-Minkowski codimension at least 4, such that the convergence is locally smooth Cheeger-Gromov sense on $\mathcal{R}$.
\end{theorem}

Furthermore, \cite[Theorem 2.18]{Bam21b} gives the characterization on metric solitons with constant Nash entropy arising as the non-collapsed $\mathbb{F}$ limit of Ricci flows. The analogue of this result for K\"{a}hler-Ricci flow was proved in \cite[Theorem 2.8]{CHM25}. These yield the following finer structure result of the tangent flow:
\begin{theorem}[Structure of the tangent flow]\label{thm:StructureOfTangentFlow}
    Suppose $M^4$ is a compact K\"{a}hler surface, $(\mathcal{X},\nu_{x_\infty;t})$ is a tangent flow obtained in \cref{thm:TangetFlow} and $\mathcal{X}=\mathcal{R}\cup\mathcal{S}$ is the regular-singular decomposition. The following hold true:
    \begin{enumerate}
        \item $\mathcal{R}$ admits a Ricci flow space time structure: there is a vector field $\partial_{\mathfrak{t}}$ on $\mathcal{R}$ such that $\partial_{\mathfrak{t}}\mathfrak{t}=1 $ and $L_{\partial_{\mathfrak{t}}}g=-2\Ricci$; 
        \item There is a K\"{a}hler orbifold with isolated singularities $(X,g)$ and a probability measure $\mu=\frac{1}{(4\pi)^2}e^{-f}$ on $X$. Let $X=\mathcal{R}_X\cup\mathcal{S}_X$ be the regular-singular decomposition of $X$. There exists an identification map $\phi:\mathcal{X}\to X\times[-\infty,0)$ such that 
        \begin{enumerate}
            \item $(X,g,f)$ is a gradient shrinking K\"{a}hler-Ricci soliton with orbifold singularities, more precisely,
            \begin{equation}\label{eq:SolitonEquation}
                \Ricci+\nabla^2f=\frac{1}{2}g, -(R+|\nabla f|^2)+f\equiv W.
            \end{equation}
            \item $\phi_t:(\mathcal{X}_t,d_t,\nu_{x_\infty;t})\to (X,\sqrt{|t|}d_g,\mu)$ is an isometry between metric measure spaces and $\phi_t$ maps $\mathcal{R}_t$ to $\mathcal{R}_X$;
            \item On $\mathcal{R}$, we have $\phi^*\frac{\partial}{\partial t}=\partial_{\mathfrak{t}}-\nabla f$, where $\frac{\partial}{\partial t}$ is the derivative with respect to the second factor in $\mathcal{R}_X\times[-\infty,0)$;
            \item There is a family of probability measures $\nu'_{x';t}, x'\in X,t\leq 0$ such that for all $s\leq t<0,x\in\mathcal{X}_t$, $(\phi_s)_*\nu_{x;s}=\nu'_{\phi_t(x);\log(s/t)}$;
            \item We have $\mu(\mathcal{S}_X)=0$. Moreover, either (i) $R\equiv0$ and $(X,d)$ is a K\"{a}hler cone with smooth link; or (ii) $R>0$ everywhere in $\mathcal{R}_X$;
            \item The complex structure and metric converges in smooth Cheeger-Gromov sense on $\mathcal{R}$: There is a precompact open exhaustion $V_i$ of $\mathcal{R}_X\times[-\infty,0)$, and embeddings $\tilde{\psi_i}:V_i\to M_k\times[-\tau_i^{-1}T,0)$, such that $\tilde{\psi_i}^*g_i(t)$ (resp. $\tilde{\psi_i}^*J$) converges to $g(t)$ (resp. $\hat{J}$) locally smoothly.
        \end{enumerate}
    \end{enumerate}
    The metric measure space $(X,d,\mu)$ is called the model space of the metric soliton $(\mathcal{X},\nu_{x_\infty;t})$. The metric soliton is uniquely determined by its model up to flow isometry.
\end{theorem}
\begin{remark}\label{rmk:CheegerGromovMaps}
    We will make use of the local smooth convergence in part (e) only at time slice $t=-1$. For later use, we suppose $U_i=V_i\cap X\times\{-1\}$ be a precompact open exhaustion of $X$ and let $\psi_i:U_i\to M_k$ be the restriction of $\phi_i$ to $U_i$. Then we have $\psi_i^*g_i(-1)$ (resp. $\psi_i^*J$) converges to $g$ (resp. $J$) locally smoothly;
\end{remark}

Note that (3e) of \cref{thm:StructureOfTangentFlow} gives a classification on the possibilities of the model space $(X,d,\mu)$. Next theorem improves this classification:
\begin{theorem}\label{thm:DichotmyOfTangentFlows}
    There are two possibilities of the model space $(X,d,\mu)$ in \cref{thm:StructureOfTangentFlow}:
    \begin{enumerate}
        \item $(X,d)$ is isometric biholomorphic to $\C^2/\Gamma, \Gamma\leq U(2)$ and 
        \[\mu=\frac{1}{(4\pi)^2}e^{-\frac{1}{4}|z|^2+\log|\Gamma|}d\vol_g\] 
        In this case, the tangent flow is the constant flow on $\C^2/\Gamma$;
        \item $(X,g,f)$ is a nonflat K\"{a}hler-Ricci shrinker with isolated orbifold singularities. In this case, the tangent flow is the associated K\"{a}hler-Ricci flow to the shrinker $(X,g,f)$.
    \end{enumerate}
\end{theorem}
\begin{proof}
    It suffices to show that the first case in (3e) of \cref{thm:StructureOfTangentFlow} implies (1). Firstly, suppose that $X=C(N)$, which is the cone over a smooth 3 manifold $M$ with metric $h$. Suppose $o$ is the cone point of $X$. We may identify $X\setminus\{o\}$ with $M\times(0,+\infty)$ and let $g=dr^2+r^2h$ be the conic metric on $X$, where $r$ is the parameter corresponding to the second factor. Direct computation shows that 
    \[\Ricci=\Ricci^M-2h, \nabla^2f=\frac{\partial^2 f}{\partial r^2}dr^2+r\frac{\partial f}{\partial r}h+(\nabla^M)^2f+\left(\frac{\partial^2f}{\partial r\partial x_i}-\frac{1}{r}\frac{\partial f}{\partial x_i}\right)dx^i\otimes_S dr.\]
    where $\Ricci^M$ is the Ricci curvature on $M$, $x_i$ are local coordinates on $M$ and $\otimes_S$ is the symmetric product. Substitute into the soliton equation
    \[\Ricci+\nabla^2f=\frac{1}{2}g\]
    and we get 
    \begin{equation}
        \begin{cases}\label{eq:SystemOfEquations}
            &\frac{\partial^2f}{\partial r^2}=\frac{1}{2}\\
            &\frac{1}{2}r^2h=\Ricci^M-2h+r\frac{\partial f}{\partial r}h+(\nabla^M)^2f\\
            &\frac{\partial^2f}{\partial r\partial x_i}-\frac{1}{r}\frac{\partial f}{\partial x_i}=0
        \end{cases}
    \end{equation}
    The first equation implies $f(x,t)=\frac{1}{4}r^2+f_1(x)r+f_0(x)$, where $f_0,f_1$ are functions on $M$. Combine with the third equation and we get $\frac{\partial f_0}{\partial x_i}=0$. Hence $f_0(x)\equiv c$ is a constant function. Plug this into the second equation and it yields
    \begin{equation}\label{eq:EquationOnM}
        \Ricci^M=2h,f_1(x)h+(\nabla^M)^2f_1(x)=0.
    \end{equation}
    The first equation above shows that $M$ is a 3 dimensional Einstein manifold with Einstein constant $2$ and hence a spherical space form. Since $X$ is K\"{a}hler, this implies that $X=\C^2/\Gamma, \Gamma\leq U(2)$. Moreover, the second equation in \cref{eq:SolitonEquation} implies that $f-|\nabla f|^2$ is a constant. It follows that $f_1(x)=c'$. Finally, by the second equation in \cref{eq:EquationOnM} we deduce that $f_1(x)=c'=0$. Therefore, $f=\frac{1}{4}r^2+c$ and $\mu=\frac{1}{(4\pi)^2}e^{-\frac{1}{4}|z|^2-c}$. Because $\mu$ is a probability measure, $c=-\log|\Gamma|$. This finishes the proof.
\end{proof}

\section{Tangent flows are not constant flows}

The main theorem of this section is 
\begin{theorem}\label{thm:TangentFlowsAreNotConstant}
    For the volume collapsing K\"{a}hler-Ricci flow $g(t), 0\leq t<T$ on $M_k$, the tangent flow at any point $x\in (M_k)_T$ is not the constant flow. Equivalently, the case (1) in \cref{thm:DichotmyOfTangentFlows} cannot occur.
\end{theorem}

As a corollary of \cref{thm:TangentFlowsAreNotConstant} and the classification \cref{thm:DichotmyOfTangentFlows}, we deduce the following description of the tangent flow:
\begin{corollary}\label{cor:TangentFlowsAreNonflatShrinkers}
    Under the assumptions of \cref{thm:TangentFlowsAreNotConstant}, any tangent flow is the associated K\"{a}hler-Ricci flow of a \textit{nonflat} gradient K\"{a}hler-Ricci shrinker $(X,g,f)$ with isolated orbifold singularities. 
\end{corollary}

To prove the above theorem, we adopt the notations from \cref{thm:TangetFlow} and let $(\mathcal{X},\nu_{x_\infty;t})$ be a tangent flow, which is the $\mathbb{F}$ limit of a blow up sequence $(M,\tau_i^{-1}g(T+\tau_it),\nu_{x,T;T+\tau_it})$. For simplicity, set
\[g_i(t)=\tau_i^{-1}g(T+\tau_it), \nu^i_{x,0;t}=\nu_{x,T;T+\tau_it}.\]

Let $\pi:M_k\to \mathbb{P}^1$ be the projection induced from $\mathcal{O}\oplus\mathcal{O}(k)\to\mathbb{P}^1$. Let $F=\pi^{-1}(y),y\in\mathbb{P}^1$ be any fiber of the projection. The lemma below calculates the evolution of the fiber volume.
\begin{lemma}\label{lem:VolumeOfTheFiber}
    We have $\vol_{g(t)}(F)=2(T-t)$ for $0\leq t<T$.
\end{lemma}
\begin{proof}
    Note that the class of the fiber $F$ is $[F]=[D_\infty]-[D_0]$. By the evolution of the class of the metric \cref{eq:EvolutionOfTheMetric},
    \begin{align*}
        &\vol_{g(t)}(F)=\int_F\omega(t)=[\omega(t)]\cdot[F]\\
        = & \left(\frac{b-t(k+2)}{k}[D_\infty]-\frac{a+t(k-2)}{k}[D_0]\right)\cdot\left([D_\infty]-[D_0]\right)\\
        = & 2\left(\frac{b-a}{2k}-t\right).
    \end{align*}
    Since $T=\frac{b-a}{2k}$, we are done.
\end{proof}

It follows immediately that for any flow $(M_k,g_i(t),\nu^i_{x,0;t})$ in the blow up sequence, the volume of the fiber $F$ is a fixed constant at time $t=-1$.

\begin{corollary}\label{cor:VolumeIsConstant}
    Under the flow $(M_i,g_i(t),\nu^i_{x,0;t})$, the following holds
    \[\vol_{g_i(-1)}(F)\equiv 2.\]
\end{corollary}

We are now ready to prove \cref{thm:TangentFlowsAreNotConstant}.

\begin{proof}[Proof of \cref{thm:TangentFlowsAreNotConstant}]
    We prove by contradiction. Assume that $(\mathcal{X},\nu_{x_\infty;t})$ is the constant flow on $X=\C^2/\Gamma$.

    Fix a point $x_0\in X\setminus\{o\}$ and consider the Cheeger-Gromov embeddings $\psi_i$ constructed in \cref{rmk:CheegerGromovMaps}. Let $F_i'=\pi^{-1}(\psi_i(x_0))\subseteq M_k$ be the fiber of $\pi$ containing $\psi_i(x_0)$ and $F_i=\psi^{-1}(F_i)$ be its preimage in $U_i$. By \cite[Theorem 7.3.8]{BBI01} and a diagonal argument, after passing to a subsequence, $F_i$ converges to a locally closed subset $F$ of $X\setminus\{o\}$ locally in the Hausdorff distance. Since $x\in F_i$, we know that $x\in F$. In particular, $F\ne\emptyset$.

    \begin{claim}\label{claim:AnalyticityOfF}
        $F$ is an analytic curve in $X\setminus\{o\}$ and hence its closure $\Bar{F}$ is an analytic curve in $X$.
    \end{claim}
    This was proved in \cite[Lemma 3.1]{CHM25}. However, since we will use the constructions of coordinates in the proof later again, we give the proof of this claim here. 
    
    Let $g_0,J_0$ be the Euclidean metric and the standard complex structure on $\C^2$. For $r$ small enough, we may choose a holomorphic coordinate chart $\varphi:B_{g_0}(0,2r)\to X\setminus\{o\}$ so that 
    \begin{enumerate}
        \item $\varphi(0)=x_0$;
        \item $\varphi^*g=g_0$;
        \item $\varphi^*J_0=\hat{J}$.
    \end{enumerate}
    Since the analyticity is a local condition, we just need to show $\varphi^{-1}(F)$ is an analytic curve.
    
    Note that the second equality holds because $g$ is a flat metric. By the local smooth convergence (cf.\cref{rmk:CheegerGromovMaps}), $\varphi^*\psi_i^*g_i(-1)$ (resp. $\varphi^*\psi_i^*J$) converges smoothly to $g_0$ (resp. $J_0$). For simplicity, set $J_i:=\varphi^*\psi_i^*J$. An application of H\"{o}rmander's $L^2$ method (cf. \cite[Lemma A.1]{CHM25}) yields the $J_i$ holomorphic coordinates $\eta_i:B_{g_0}(0,r)\to B_{g_0}(0,2r)$ such that $\eta_i$ converges smoothly to the inclusion $B_{g_0}(0,r)\hookrightarrow B_{g_0}(0,2r)$. By construction we know that $\varphi^{-1}(F_i)$ is $J_i$ holomorphic in $B_{g_0}(0,2r)$ and hence $(\varphi\circ\eta_i)^{-1}(F_i)$ is $J_0$ holomorphic in $B_{g_0}(0,r)$. Since $\varphi$ is an isometry, $\varphi^{-1}(F_i)$ converges to $\varphi^{-1}(F)$ in Hausdorff distance. Furthermore, since $\eta_i$ converges to the inclusion smoothly, we deduce that $(\varphi\circ\eta_i)^{-1}(F_i)$ converges to $\varphi^{-1}(F)$ in $B_{g_0}(0,r)$. By the proof of \cite[Theorem 1]{Bis64}, $\varphi^{-1}(F)$ is an analytic curve. This finishes the proof of \cref{claim:AnalyticityOfF}.

    \begin{claim}\label{claim:VolumeIsLarge}
        For any positive integer $N$ and small $r>0$, there exists an $i$ such that 
        \[\vol_{g_i(-1)}(F_i')\ge cNr^2\]
        where $c$ is some universal constant.
    \end{claim}
    Since $X$ is a quotient by $\C^2$ and any analytic curve in $\C^2$ is unbounded, we that know $F$ is unbounded. Therefore, we can pick points $x_\alpha\in F,\alpha=1,\cdots,N$ and a small $r$ so that $B_g(x_\alpha,2r)\subseteq X\setminus\{o\}$ and $B_g(x_\alpha,2r)\cap B_g(x_\beta,2r)=\emptyset, \alpha\ne\beta$. For each $\alpha$, we may construct a holomorphic coordinates $\varphi^\alpha:B_{g_0}(0,2r)\to B_g(x_\alpha,2r)$ and $J_i$ holomorphic coordinates $\eta_i^\alpha:B_{g_0}(0,r)\to B_{g_0}(0,2r)$ as in the proof of \cref{claim:AnalyticityOfF}. Therefore, $(\varphi\circ\eta^\alpha_i)^{-1}(F_i)$ is an analytic curve in $B_{g_0}(0,r)$. It follows from \cite[Lemma 3]{Bis64} that 
    \[\Hau^2((\varphi\circ\eta^\alpha_i)^{-1}(F_i)\cap B_{g_0}(0,\frac{r}{2}))\ge c_1r^2\]
    where $c_1$ is a universal constant. Since $(\psi\circ\varphi\circ\eta^\alpha_i)^*g_i(-1)$ converges smoothly to $g_0$, for large enough $i$ and any $\alpha$, 
    \[\vol_{g_i(-1)}(F_i'\cap(\psi\circ\varphi\circ\eta^\alpha_i)(B_{g_0}(0,\frac{r}{2})))\ge \frac{c_1r^2}{2}.\]
    Since the images $(\psi_i\circ\varphi\circ\eta^\alpha_i)(B_{g_0}(0,\frac{r}{2}))$ are disjoint subsets in $M_k$ for large enough $i$, there holds
    \[\vol_{g_i(-1)}(F_i')\ge\sum_{\alpha=1}^N\vol_{g_i(-1)}(F_i'\cap(\psi\circ\varphi\circ\eta^\alpha_i)(B_{g_0}(0,\frac{r}{2})))\ge \frac{Nc_1r^2}{2}.\]
    This finishes the proof of \cref{claim:VolumeIsLarge}.
    
    Finally, by choosing $N$ large enough, we would have 
    \[\vol_{g_i(-1)}(F_i')>2\]
    for some $i$, which contradicts \cref{cor:VolumeIsConstant}. Therefore, we have proved \cref{thm:TangentFlowsAreNotConstant}.
\end{proof}

\section{Finiteness of Singularities}
\cref{cor:TangentFlowsAreNonflatShrinkers} gives a first step towards understanding the tangent flow, namely, it is modeled on a nonflat gradient K\"{a}hler-Ricci shrinker $(X,g,f)$ with isolated orbifold singularities. However, a priori there may be infinitely many such singularities. We rule out this possibility and thus finishes the proof of \cref{thm:MainTheorem} in this section. 

\begin{theorem}\label{thm:FiniteSingularities}
    In \cref{cor:TangentFlowsAreNonflatShrinkers}, $X$ has only finitely many singularities.
\end{theorem}

The rest of this section is devoted to the proof of \cref{thm:FiniteSingularities}, which is purely topological. The proof is in spirit of those in \cite[Theorem 12, Proposition 13]{CFSZ20} and \cite[Proposition 5.4]{CHM25}, with some modifications. The idea is as follows: each orbifold singularity $x$ of $X$ will give rise to an embedded 4 manifold in the original K\"{a}hler surface $M_k$ whose boundary is diffeomorphic to some nontrivial spherical space form $\mathbb{S}^3/\Gamma$. If there were infinitely many singularities in $X$, we would be able to construct an arbitrarily large number of such manifolds embedded in $M_k$, which are pairwisely disjoint. Combine with the smooth convergence of K\"{a}hler structure, we will then arrive at a contradiction by using the standard result in the theory of symplectic filling.

\begin{lemma}\label{lem:BallsNearSingularities}
    If $x\in X$ is a singularity in $X$ of type $\C^2/\Gamma, \Gamma\leq U(2)$, then there is $r_0>0$ and a sequence of points $y_i\in M_k$, such that for all $r<r_0$, we have $\partial B_{g_i(-1)}(y_i,r)\cong \mathbb{S}^3/\Gamma$ for $i$ large enough. Consequently, $B_{g_i(-1)}(y_i,r)$ is a 4 manifold whose boundary is diffeomorphic to $\mathbb{S}^3/\Gamma$. 
\end{lemma}
\begin{proof}
    Choose $r_0$ small so that $B_g(x,r)$ is diffeomorphic to $B_1(0)/\Gamma$ for all $r<r_0$. As in the proof of \cref{thm:TangentFlowsAreNotConstant}, let $\psi_i:U_i\to M_k$ be the sequence of Cheeger-Gromov embeddings. We may choose $x_i\in U_i$ converging to $x$ and let $y_i=\psi_i(x_i)$. Therefore, we have $(M_i,g_i(-1),y_i)$ converges to $(X,g,x)$ in 
    the pointed orbifold Cheeger-Gromov sense (see \cite[Definition 3.1]{HM11} for the precise definition). In particular, we have $B_{g_i(-1)}(y_i,r_0)$ converges to $B_g(x,r_0)$ in Gromov-Hausdorff sense. It follows that $\partial B_{g_i(-1)}(y_i,r)\in M_k\setminus B_{g_i(-1)}(y_i,\varepsilon)$ for $i$ large enough and $\varepsilon<r$. Since we have locally smooth convergence  away from $x$, $\partial B_{g_i(-1)}(y_i,r)$ is diffeomorphic to $\partial B_{g}(x,r)\cong\mathbb{S}^3/\Gamma$ for $i$ large enough.
\end{proof}

For the purpose of contradiction, we assume that there are infinitely many orbifold singularities in $X$. Fix an arbitrary integer $N$. By the assumption there is a compact set $K$ containing at least $N$ singularities in $X$, which we denote with $x_n,n=1,2,\cdots, N$. Suppose that $x_n$ is of type $\C^2/\Gamma_n$, where $\Gamma_n\leq U(2)$ acts freely on $\C^2\setminus\{(0,0)\}$ and is nontrivial. It follows from \cref{lem:BallsNearSingularities} that there is $r>0$ and $y_{i,n}$, such that $B_{g_i(-1)}(y_{i,n},r)$ are also pairwise disjoint balls in $(M_k,g_i(-1))$ and $\partial B_{g_i(-1)}(y_{i,n},r)\cong\mathbb{S}^3/\Gamma_n$ for any $i>i_0$, where $i_0$ is some large integer. For simplicity, we will denote $B_{g_i(-1)}(y_{i,n},r)$ with $B_{i,n}$ from now on.

\begin{lemma}(cf.\cite[Proposition 13]{CFSZ20})\label{lem:ExistenceOfSpecialBn}
    Fix $N$ large enough, then for any $i>i_0$, there is some ball $B_{i,n(i)}$ such that
    \begin{equation}\label{eq:HomologyOfTheBoundary}
        H_2(\partial B_{i,n(i)},\Z)=H_2(\mathbb{S}^3/\Gamma_n,\Z)=A\oplus A
    \end{equation}
    where $A$ is an abelian group and that
    \begin{equation}\label{eq:TrivialSecondHomology}
        H_2(B_{i,n(i)},\Z)=0.
    \end{equation}
\end{lemma}
\begin{proof}
    Let $\iota_{i,n}:\partial B_{i,n}\to B_{i,n}$ be the inclusion map. We firstly show that there is some $n(i)$, such that 
    \begin{equation}
        (\iota_{i,n(i)})_*:H_m(B_{i,n(i)},R)\to H_m(B_{i,n(i)},R), m=0,1,2,3
    \end{equation}
    is surjective. Here $R$ is any coefficient ring.
    The case $m=0$ is trivially true for all $B_{i,n}$. Now we prove the case when $m=1,2,3$. Let $B_i=\cup_{n=1}^NB_{i,n(i)}$ be the disjoint union of all $N$ balls and $\iota_i:\partial B_i\to B_i$ be the inclusion map. Consider the Mayer-Vietoris sequence associated to the decomposition $M_k=B_i\cup\overline{M_k\setminus B_i}$:
    \begin{equation}
        H_m(\partial B_i,R)\xrightarrow{(\iota_i)_*} H_m(B_i,R)\oplus H_m(\overline{M_k\setminus B_i},R)\xrightarrow{(\alpha_i)_*-(\beta_i)_*}H_m(M_k,R)\to H_{m-1}(\partial B_i,R)
    \end{equation}
    where $\alpha_i$ and $\beta_i$ are inclusions from $B_i$ and $\overline{M_k\setminus B}$ to $M_k$ respectively. Observe that if $\gamma\notin\im((\iota_i)_*)$, then $(\alpha_i)_*\gamma\notin \im((\beta_i)_*)$. Therefore, there is an injection:
    \[H_m(B,R)/\im((\iota_i)_*)\hookrightarrow H_m(M_k,R)/\im((\beta_i)_*).\]
    Note that the LHS is $\oplus_{n=1}^NH_m(B_{i,n},R)/\im((\iota_{i,n})_*)$. Since the right hand side has a bounded number of generators, if $N$ is chosen large enough, there must be some $B_{i,n(i)}$ such that $\iota_{i,n(i)}:H_m(\partial B_{i,n},R)\to H_m(B_{i,n},R)$ is surjective for $m=1,2,3$.
    The rest of the proof is the same as that in \cite[Proposition 13]{CFSZ20}.
\end{proof}

Now we can finish the proof of \cref{thm:FiniteSingularities}.
\begin{proof}[Proof of \cref{thm:FiniteSingularities}]
    By \cref{lem:ExistenceOfSpecialBn}, there is a ball $B_{i,n(i)}$ satisfying \cref{eq:HomologyOfTheBoundary} and $\cref{eq:TrivialSecondHomology}$ in $(M_k,g_i(-1))$. By \cite[Lemma 6.3]{Zha19}, \cref{eq:HomologyOfTheBoundary} implies that $\Gamma_n$ is either binary dihedral group $D_{2n}$ or binary isocahedral group. In particular, $\Gamma_n\leq SU(2)$ and hence $\mathbb{S}^3/\Gamma$ is the link of simple surface singularity. Let $\xi$ be the contact structure on $\mathbb{S}^3/\Gamma$ induced by the standard contact structure on $\mathbb{S}^3$. By the smooth Cheeger-Gromov convergence of K\"{a}hler metric, $(B_{i,n(i)},\omega_i(-1))$ is a symplectic filling of $(\mathbb{S}^3/\Gamma,\xi)$ for $i$ large enough (cf.\cite[Proposition 5.4]{CHM25}). Furthermore, \cref{eq:TrivialSecondHomology} implies that there is no symplectic sphere of self intersection $(-1)$ in $B_{i,n(i)}$ and hence $B_{i,n(i)}$ is a minimal symplectic filling of $\mathbb{S}^3/\Gamma_n$. Note that the minimal resolution of the singularity $\C^2/\Gamma_n$ is also a symplectic filling of $\mathbb{S}^3/\Gamma_n$. Therefore, by the uniqueness of the diffeomorphism type of symplectic filling (\cite[Main Theorem]{OO05}), $B_n$ is diffeomorphic to the minimal resolution of $\C^2/\Gamma_n$. In particular, $B_{i,n(i)}$ contains a sphere of self intersection $(-2)$ (cf.\cite[III.7]{BHPVdV04}), which contradicts \cref{eq:TrivialSecondHomology}.
\end{proof}

\nocite{*}
\bibliographystyle{alpha_name-year-title}
\bibliography{References}
\end{document}